\documentclass[12pt,reqno]{article}
\usepackage[margin=1in]{geometry}
\usepackage[all]{xy}
\usepackage[utf8]{inputenc}
\usepackage{pdfpages}
\usepackage{amsmath}
\usepackage{amssymb}
\usepackage{csquotes}
\usepackage{amsthm}
\usepackage{enumerate}
\usepackage{textcomp}
\usepackage{todonotes}
\usepackage[titletoc,toc,title]{appendix}
\usepackage[colorlinks,citecolor={blue},linkcolor={blue},linktocpage=true]{hyperref}
\usepackage{makeidx}
\usepackage[refpage]{nomencl}[0]
\usepackage{mathtools}
\usepackage[makeroom]{cancel}
\usepackage{graphicx,nicefrac}
\usepackage{authblk}
\usepackage[symbol*]{footmisc}
\usepackage[makeroom]{cancel}
\usepackage{indentfirst}
\usepackage{fancyhdr}
\usepackage{nameref,cleveref}
\usepackage{babel}

\crefname{thm}{theorem}{theorems}
\Crefname{thm}{Theorem}{Theorems}
\newcommand{\Addresses}{{
		\bigskip
		\footnotesize
		
		\textsc{School of Mathematical Sciences, Tel Aviv University, Ramat Aviv, Tel Aviv 6997801, Israel}\par\nopagebreak
		\textit{E-mail address:} \href{mailto:zahihaza@tauex.tau.ac.il}{\color{black}{\textbf{zahihaza@tauex.tau.ac.il}}}

}}
\makeatletter

\def\mynameref#1#2{%
	\begingroup
	\edef\@mytxt{#2}%
	\edef\@mytst{\expandafter\@thirdoffive\@mytxt}%
	\ifx\@mytst\empty\else
	\space(\nameref{#1})\fi
	\endgroup
}

\makeatother

\newcommand{\ocal}{\mathcal{O}}
\newcommand{\pcal}{\mathcal{P}}
\newcommand{\pomega}{\varpi}
\newcommand{\bbc}{\mathbb{C}}
\newcommand{\bbr}{\mathbb{R}}
\newcommand{\bbz}{\mathbb{Z}}

\theoremstyle{plain}
\newtheorem{Theorem}{Theorem}

\newtheorem{seclem}{Lemma}
\newtheorem{intseclem}{Lemma}

\newtheorem{cor}{Corollary}
\newtheorem{prop}{Proposition}
\newtheorem{intTheorem}{Theorem}

\theoremstyle{definition}

\theoremstyle{remark}

\title{A Note on the Asymptotic Expansion of Matrix Coefficients over $p$-adic Fields}
\author{Zahi Hazan}
\date{}

\begin{document}
\maketitle

\abstract{
In this note, presented as a ``community service", followed by the PhD research of the author, we draw the relation between Casselman's theorem \cite{casselman1993introduction} regarding the asymptotic behavior of matrix coefficients of reductive algebraic groups over $p$-adic fields and its expression as a finite sum of finite functions. In addition, we write the expansion explicitly for general linear groups.
}

\section{Introduction}
Let $k$ be a non-Archimedean locally compact field with $\ocal$, its ring of integers, and $\pcal$, the maximal ideal in $\ocal$. Denote a uniformizer of $\pcal$ by $\pomega$, and the cardinality of the residue field by $q$. Let $G$ be the group of $k$-rational points of a $k$-split reductive algebraic group. Let $(\pi,V)$ be a (complex) smooth, admissible, irreducible representation of $G$. For a parabolic subgroup $P$ of $G$ with Levi decomposition $P=MN$, we denote by $V(N)$ the subspace of $V$ generated by
\begin{equation*}
\left\{ \pi(n)v-v |\ n\in N,\ v\in V\right\}.
\end{equation*}
We also denote $V_{N}=V/V(N)$. This is the space of a smooth representation $\left(\pi_{N},V_{N}\right)$ of $M$, called the Jacquet module of $\left(\pi,V\right)$ along $N$.

Let $P_\phi=M_\phi N_\phi$ be a minimal parabolic subgroup.  Let $T_\phi=Z\left(M_\phi\right)$ be the center of $M_\phi$. This is the maximal split torus of $G$. Let $\Sigma$ be the set of all roots corresponding to the pair $\left(G,T_\phi\right)$. i.e., the non-trivial eigencharacters of the adjoint action of $T_\phi$ on the lie algebra $\mathfrak{g}$ of $G$. Let $\Sigma^+$ be the subset of positive roots determined by $P_\phi$, so that $\mathfrak{n}=\bigoplus_{\alpha\in \Sigma ^+}\mathfrak{g}_\alpha$, where $\mathfrak{g}_\alpha$ is the eigenspace of $\alpha$, and $\mathfrak{n}$ is the lie algebra of $N_\phi$. Let $\Delta$ be the basis of $\Sigma^+$, so that every root in $\Sigma^+$ is a sum of roots in $\Delta$.

For $\Theta\subseteq \Delta$, let $T_\Theta$ be the connected component of the identity in $\bigcap _{\alpha \in \Theta} \ker(\alpha)$, $M_\Theta= Z_G(T_\Theta)$ the centralizer of $T_\Theta$ in $G$ (so $T_\Theta=Z(M_\Theta)$ the center of $M_\Theta$), and $P_\Theta=M_\Theta N_\Theta$ the standard parabolic corresponding to $\Theta$, where $N_\Theta$ its unipotent radical.

For each $\Theta\subseteq \Delta$ and $0<\varepsilon\leq 1$, we define
\begin{equation*}
{_{\Theta}T_{\phi}^{-}}\left(\varepsilon\right)=\left\{ a\in T_\phi\left| \begin{array}{l}
\left|\alpha\left(a\right)\right| \leq\varepsilon,\ \forall\alpha\in\Delta\backslash \Theta \\ \varepsilon < \left|\alpha\left(a\right)\right|\leq 1,\ \forall\alpha\in \Theta
\end{array} \right.\right\},
\end{equation*}
and $T_{\Theta}^{-}\left(\varepsilon\right)={_{\Theta}T_{\phi}^{-}}\left(\varepsilon\right)\cap T_{\Theta}$. These are subsets of
\begin{equation*}
	T_\phi^-=\left\{a\in T_\phi|\ \left|\alpha(a)\right|\leq 1,\ \forall \alpha\in \Delta\right\}.
\end{equation*}
For any $0<\varepsilon \leq 1$ one sees that $T_\phi^-$ is the disjoint union of the ${_{\Theta}T_{\phi}^{-}}\left(\varepsilon\right)$ as $\Theta$ ranges over all subsets of $\Delta$.	
Moreover, by \cite[Proof of Lemma 2.3.2]{beuzart2012conjecture} we have the following decomposition of $T_\phi^-$.
\begin{intseclem}\label{torus_decomp_lem}
	Let $v\in V$. Let $K_v$ be an open subgroup of
	\begin{equation*}
		T_1 = \left\{a\in T_\phi \left| |\alpha(a)|=1,\ \forall \alpha\in \Delta\right. \right\},
	\end{equation*}
	that stabilizes $v$. For all $0<\varepsilon< 1$,
	\begin{equation}\label{cone_decomp}
		{_{\Theta}T_{\phi}^{-}}\left(\varepsilon\right)=\bigcup_{\gamma\in \Gamma_\Theta^\varepsilon}T_{\Theta}^{-}\left(\varepsilon\right)\gamma K_v,
	\end{equation}
	where $\Gamma_\Theta^\varepsilon\subseteq T_\phi^-$ is a finite set.
	Hence,
	\begin{equation*}
	T_\phi^-=\bigcup _{\Theta\subseteq \Delta }\bigcup_{\gamma\in \Gamma_\Theta^\varepsilon}T_{\Theta}^{-}\left(\varepsilon\right)\gamma K_v.
	\end{equation*}
\end{intseclem}

Let $N^{-}_\Theta$ be the unipotent radical opposite to $N_\Theta$. We define a canonical non-degenerate pairing of $V_{N_\Theta}$ with $\tilde{V}_{N^{-}_\Theta}$ according to the formula
\begin{equation*}
\langle u_\Theta,\tilde{u}_\Theta\rangle_{N_\Theta}=\langle v,\tilde{v}\rangle,
\end{equation*}
where $v\in V,\tilde{v}\in\tilde{V}$ are any two canonical lifts of $u_\Theta,\tilde{u}_\Theta$.
We have the following theorem from \cite[Corollary 4.3.4]{casselman1993introduction}.
\begin{intTheorem}[Casselman]\label{Theorem:Casselman}
	Let $v\in V$ and $\tilde{v}\in\tilde{V}$ be given. For $\Theta\subseteq \Delta$, let $u_\Theta,\ \tilde{u}_\Theta$ be their images in $V_{N_\Theta},\ \tilde{V}_{N^{-}_\Theta}$. There exists $\varepsilon>0$ such that for any $\Theta\subseteq \Delta$ and  $a\in {_{\Theta}T_{\phi}^{-}}(\varepsilon)$ one has 
\begin{equation*}
	\left\langle \pi(a)v,\tilde{v}\right\rangle =\left\langle \pi_{N_\Theta}(a)u_\Theta,\tilde{u}_\Theta\right\rangle _{N_\Theta}.
\end{equation*}

\end{intTheorem}

For a subgroup $H$ of $G$, we say that a function is an $H$-finite function (or simply finite) if the space spanned by its right $H$-translations is finite dimensional.
By Jacquet and Langlands \cite[Lemma 8.1]{jacquet2006automorphic} we have an explicit basis for all $H$-finite function when $H$ is a locally compact abelian group.
\begin{intTheorem}[Jacquet and Langlands]\label{Theorem:jac_lang_thrm}
	Let $H$ be a locally compact abelian group of the form 
	\begin{equation*}
		H=K\times\bbz^{r}\times\bbr^{n}
	\end{equation*}
	where $K$ is a compact group. For $1\leq i\leq r+n$ let $\xi_{i}:\left(h_{0},x_{1},\ldots,x_{r+n}\right)\to x_{i}$ be the projection map. Then, for any sequence of non-negative integers $p_{1},\ldots,p_{r+n}$ and any quasi-character $\chi$ of $H$, the function $\chi\prod_{i=1}^{r+n}\xi_{i}^{p_{i}}$ is continuous and finite. These functions form a basis of the space of continuous finite functions on $H$.
\end{intTheorem}

Our main goal is to show how \Cref{Theorem:Casselman} and \Cref{Theorem:jac_lang_thrm} give an asymptotic expansion for matrix coefficients in terms of a finite linear combination of $T_\Theta$-finite functions.
In detail, 

\begin{Theorem}\label{asym_exp_for_all_prop}	
	Let $v\in V$ and $\tilde{v}\in\tilde{V}$. There exists $\varepsilon>0$ such that for any $\Theta\subseteq \Delta$ and  $a\in {_{\Theta}T_{\phi}^{-}}(\varepsilon)$ there exist finite sets of vectors, that depend on $\{\pi, v, \tilde{v}\}$, $\underline{p'}=\left(p'_{1},\ldots,p'_{r}\right)\in \bbr^r,\ \underline{p}=\left(p_{1},\ldots,p_{r}\right)\in \bbz^r_{\ge 0}$, and $\underline{\chi}=\left(\chi_{1},\ldots,\chi_{r}\right)$ where for all $1\leq i\leq r$, $\chi_i:k^\times \to \bbc^\times$ is a unitary character, so that
	\begin{equation}\label{asym_exp_p_adic}
	\left\langle \pi(a)v,\tilde{v}\right\rangle = \sum_{\underline{p'},\underline{p},\underline{\chi}}\alpha_{\underline{p'},\underline{p},\underline{\chi}}\prod_{i=1}^{r}\chi_i(b_i)\left|b_{i}\right|^{p'_{i}}\log_q^{p_{i}}\left|b_{i}\right|,
	\end{equation}
	where $r$ is such that $T_{\Theta}^{-}\left(\varepsilon\right)\cong \left(k^\times\right)^r$, $b\in T_{\Theta}^{-}\left(\varepsilon\right)$ identified with $(b_1,\ldots,b_r)$ and corresponding to $a$ as in \Cref{torus_decomp_lem}, and $\alpha_{\underline{p'},\underline{p},\underline{\chi}}\in \bbc$ are such that $\alpha_{\underline{p'},\underline{p},\underline{\chi}}=0$ for all but finitely many  $\underline{p'},\underline{p},\underline{\chi}$.
\end{Theorem}

\Cref{main_thm_proof_sec} is devoted to the proof of the following proposition.
\begin{prop}\label{prop:asym_exp}
	Let $\Theta\subseteq \Delta$, $u_\Theta\in V_{N_\Theta}$ and $\tilde{u}_\Theta\in\tilde{V}_{N_\Theta^-}$. Let $r$ be such that $T_{_\Theta}\cong \left(k^\times\right)^r$ by the map $b\mapsto\left(b_1,\ldots, b_r\right)$. There exist finite sets of vectors, that depend on $\{\pi, u_\Theta, \tilde{u}_\Theta\}$, $\underline{p'}=\left(p'_{1},\ldots,p'_{r}\right)\in \bbr^r,\ \underline{p}=\left(p_{1},\ldots,p_{r}\right)\in \bbz^r_{\ge 0}$, and $\underline{\chi}=\left(\chi_{1},\ldots,\chi_{r}\right)$, where for all $1\leq i\leq r$, $\chi_i:k^\times \to \bbc^\times$ is a unitary character, such that for all $b\in T_{_\Theta}$, one has
	\begin{equation*}
	\left\langle \pi_{N_\Theta}(b)u_\Theta,\tilde{u}_\Theta\right\rangle_{N_\Theta} =\sum_{\underline{p'},\underline{p},\underline{\chi}}\alpha_{\underline{p'},\underline{p},\underline{\chi}}\prod_{i=1}^{r}\chi_i(b_i)\left|b_{i}\right|^{p'_{i}}\log_q^{p_{i}}\left|b_{i}\right|,
	\end{equation*}
	where $\alpha_{\underline{p'},\underline{p},\underline{\chi}}\in \bbc$ such that $\alpha_{\underline{p'},\underline{p},\underline{\chi}}=0$ for all but finitely many  $\underline{p'},\underline{p},\underline{\chi}$.
\end{prop}

The proof of \Cref{asym_exp_for_all_prop} is followed immediately from \Cref{prop:asym_exp} and \Cref{torus_decomp_lem}. Indeed,
\begin{proof}[Proof of \Cref{asym_exp_for_all_prop}]
	Let $\Theta\subseteq \Delta$ and  $a\in {_{\Theta}T_{\phi}^{-}}(\varepsilon)$ . Using the decomposition \eqref{cone_decomp}, for any $0<\varepsilon\leq 1$, there exist $b\in T_{\Theta}^{-}\left(\varepsilon\right)$ (corresponding to $a$ as in \Cref{torus_decomp_lem}), $\gamma\in T_\phi^-$, and $k_1\in K_v$, such that we can write $a=b\gamma k_1$. Hence,
	\begin{equation*}
	\left\langle \pi(a)v,\tilde{v}\right\rangle=\left\langle \pi(b\gamma k_1)v,\tilde{v}\right\rangle=\left\langle \pi(b\gamma)v,\tilde{v}\right\rangle.
	\end{equation*}
	Since $b\in T_{\Theta}^{-}\left(\varepsilon\right)$ and $\gamma \in T_\phi^-$, we have $b\gamma \in T_{\Theta}^{-}\left(\varepsilon\right)$. Thus, we can apply \Cref{Theorem:Casselman}, i.e. there exists $\varepsilon>0$ such that 
	\begin{equation*}
	\left\langle \pi(b\gamma)v,\tilde{v}\right\rangle=\left\langle \pi_{N_\Theta}(b\gamma)u_\Theta,\tilde{u}_\Theta\right\rangle _{N_\Theta}.
	\end{equation*}
	We  write
	\begin{equation*}
	\left\langle \pi_{N_\Theta}(b\gamma)u_\Theta,\tilde{u}_\Theta\right\rangle _{N_\Theta}=\left\langle \pi_{N_\Theta}(b)\left(\pi_{N_\Theta}(\gamma)u_\Theta\right),\tilde{u}_\Theta\right\rangle_{N_\Theta}.
	\end{equation*}
	We have $b\in T_{\Theta}^{-}\left(\varepsilon\right)\subseteq T_\Theta$, so the result is obtained by applying \Cref{prop:asym_exp}, where $u_\Theta$ is replaced with $\pi_{N_\Theta}(\gamma)u_\Theta$.
\end{proof}

We note that \Cref{torus_decomp_lem} does not give an explicit expression of $b\in T_{\Theta}^{-}\left(\varepsilon\right)$ in terms of its correspondent $a\in {_{\Theta}T_{\phi}^{-}}(\varepsilon)$. In \Cref{gln_case_sec} we give a constructed proof of \Cref{torus_decomp_lem} for the case $G=\mathrm{GL}_n$ for any positive integer $n$. This allows us to write an explicit asymptotic expansion of the matrix coefficient at $a$ in terms of its coordinates.

\section{Proof of \Cref{prop:asym_exp}}\label{main_thm_proof_sec}

First, we use \Cref{Theorem:jac_lang_thrm} to write an explicit basis for all finite function in our case,
\begin{cor}\label{cor:ourfinitefunc}
	Let $\Theta \subseteq \Delta$ and $f:T_\Theta\to\bbr$ a continuous finite function. Let $r$ be such that $T_{\Theta}\cong \left(k^\times\right)^r$ by the map $b\mapsto\left(b_1,\ldots, b_r\right)$. Then, for all $b\in T_\Theta$, $f(b)$ is spanned by
	\begin{equation*}
		\prod_{i=1}^{r}\chi_i(b_i)\left|b_i\right|^{p'_i}\log_q ^{p_i}\left|b_i\right|,
	\end{equation*}	
	where $\left(p'_{1},\ldots,p'_{r}\right)\in \bbr^r,\ \left(p_{1},\ldots,p_{r}\right)\in \bbz^r_{\ge 0}$, and for all $1\leq i\leq r$, $\chi_i:k^\times \to \bbc^\times$ is a unitary character.
\end{cor}
\begin{proof}
	We have $k^{\times}\cong\ocal^{\times}\times\bbz$ by the map $x\mapsto \left(\frac{x}{\left|x\right|},\log_q\left|x\right|\right)$. A character of $k^\times$ is of the form $\chi'(\cdot) \left|\cdot \right|^s$, where $\chi'$ is a unitary character and $0\neq s\in \bbc$. We can assume $s\in\bbr$ by attaching the imaginary part to $\chi'$. Let $(b_1,\ldots,b_r)\in  \left(k^\times\right)^r$ be the image of $b\in T_\Theta$. Let $\chi$ be a quasi character of $\left(\ocal^\times\right)^r$. Then $\chi(b)=\prod_{i=1}^{r}\chi_i(b_i)\left|b_i\right|^{p'_i}$, where  $\chi_i$ is unitary and $p'_i\in \bbr$ for all $1\leq i\leq r$.
	Applying \Cref{Theorem:jac_lang_thrm} for $H=\left(k^\times\right)^r$ implies that the space of continuous finite function on $T_\Theta$ is spanned by $\chi\prod_{i=1}^{r}\xi_{i}^{p_{i}}$, where $\chi$ is a quasi character of $\left(\ocal^\times\right)^r$ and for all $1\leq i\leq r$, $\xi_i$ is the projection map to each coordinate of $\bbz^r$, i.e. $\xi_i(b)=\log_q\left|b_i\right|$. Therefore, this space is spanned by $\prod_{i=1}^{r}\chi_i(b_i)\left|b_i\right|^{p'_i}\log_q\left|b_i\right|^{p_i}$.
\end{proof}

In order to deduce \Cref{prop:asym_exp} from  \Cref{Theorem:Casselman} and \Cref{cor:ourfinitefunc}, it is left to show that for each $\Theta\subseteq \Delta$, the function $x\mapsto\left\langle \pi_{N_\Theta}(x)u_\Theta,\tilde{u}_\Theta\right\rangle _{N_\Theta}$ is a $T_\Theta$-finite function. Namely,

\begin{prop}\label{prop:The-space-spanned}
	Let $u_\Theta\in V_{N_\Theta}$, $\tilde{u}_\Theta\in\tilde{V}_{N_\Theta^{-}}$. There exist finite sets $\{b_i\}_{1\leq i\leq \ell}\subseteq  T_\Theta$ and $\{c_i(b)\}_{1\leq i\leq \ell}\subseteq \bbc$, such that for all $b\in T_\Theta$ and $m\in M_\Theta$ we have
	\begin{equation*}
		\left\langle \pi_{N_\Theta}(mb)u_\Theta,\tilde{u}_\Theta\right\rangle _{N_\Theta}=\sum_{i=1}^\ell c_i(b)\left\langle \pi_{N_\Theta}(m b_i)u_\Theta,\tilde{u}_\Theta\right\rangle _{N_\Theta}.
	\end{equation*}
\end{prop}

Before proving \Cref{prop:The-space-spanned} we need the following lemma.

\begin{seclem}\label{seclem:cent_gen_fin}
	Let $R$ be a group with center $Z\left(R\right)\cong K\times\bbz^{r}$, where $K$ is a compact group. Denote the standard basis of $\bbz^{r}$ by $\left\{ e_{1},\ldots,e_{r}\right\} $. Let $\left(L,\sigma\right)$ be a (complex) smooth $R$-module of finite length.
	\begin{enumerate}
		\item[(i.)] For all finite dimensional spaces $W\subseteq L$ and all $1\leq j \leq r$ there exists a finite dimensional $Z(R)$-invariant space $W_j\subseteq L$, such that
		\begin{equation}\label{sigma_e_j_W}
		\sigma(e_j)W\subseteq  W+W_j.
		\end{equation}
		\item[(ii.)] For all finite dimensional spaces $W\subseteq L$, there exists a finite dimensional $Z(R)$-invariant space $W'\subseteq L$, such that
		\begin{equation*}
		\sigma(\bbz^r)W\subseteq  W+W'.
		\end{equation*}
		\item[(iii.)] Let $v\in L$. The $Z\left(R\right)$-module generated by $v$ is finite dimensional.
	\end{enumerate} 
\end{seclem}

\begin{proof}
	We begin by proving part (i.).  It is sufficient to prove this part for a one dimensional space $W$ as the general case follows directly. Hence, we assume that $W$ is spanned as a $R$-module by $v\in L$. We prove this part by induction on the length of $L$. First, assume the length is $1$. i.e. $L$ is irreducible. Then, by Schur's lemma, $Z\left(R\right)$ acts on $L$ as a scalar. Thus, the $Z\left(R\right)$-module generated by $v$ is of dimension $1$ and all the parts of the lemma follow. Now, assume that the assertion is true for modules of length $d-1$. Let $L$ be a $R$-module of length $d$. That is, there exists a sequence of $R$-modules
	\begin{equation*}
	0=L_{0}\subsetneq L_{1}\subsetneq L_{2}\subsetneq\ldots\subsetneq L_{d}=L
	\end{equation*}
	such that $L_{i+1}/L_{i}$ is irreducible for all $0\leq i<d$.
	
	By the fact that $L_{d}/L_{d-1}$ is irreducible, and by Schur's lemma, for all $1\leq j\leq r$ there exists $\alpha_j\in\bbc$, such that 
	\begin{equation*}
	\sigma\left(e_{j}\right)\left(v+L_{d-1}\right)=	\alpha_j v +L_{d-1}.
	\end{equation*}
	In particular, 
	\begin{equation}\label{sigma_e_j_v}
	\sigma\left(e_{j}\right)v=\alpha_{j}v+h_{j},
	\end{equation}
	where $h_{j}\in L_{d-1}$.

	Let $w=\sum_{i=1}^{\ell}c_i\sigma(g_i) v\in W$, where $c_i\in \bbc$ and $g_i\in R$ for all $1\leq i \leq \ell$. Then, by \cref{sigma_e_j_v}
	
	\begin{equation}\label{sigma_e_j_w}
		\sigma\left(e_{j}\right)w=	\sigma\left(e_{j}\right)\left(\sum_{i=1}^{\ell}c_i\sigma(g_i) v\right)=\sum_{i=1}^{\ell}c_i\sigma(g_i)\sigma\left(e_{j}\right)v=\sum_{i=1}^{\ell}c_i\sigma(g_i)\left(\alpha_{j}v+h_{j}\right).
	\end{equation}
	Denote by $U_j$ the $R$-module spanned by $h_{j}$. In this notation, \cref{sigma_e_j_w} gives
	\begin{equation}\label{sigma_e_j_v2}
	\sigma(e_j)w\in \bbc w+U_j.
	\end{equation}
	
	We have $U_j\subseteq L_{d-1}$. Thus, by the induction hypothesis, there exists a finite dimensional $Z(R)$-invariant space $W'_j\subseteq L_{d-1}$, such that
	\begin{equation}\label{sigma_e_j_U_j}
	\sigma(e_j)U_j\subseteq  U_j+W'_j.
	\end{equation}
		
	We take $W_j=U_j+W'_j$. Thus, $U_j\subseteq W_j$, so \cref{sigma_e_j_v2} gives $\sigma(e_j)w\in \bbc w+W_j$ and by \cref{sigma_e_j_U_j} gives 
	\begin{equation*}
		\sigma(e_j)W_j=\sigma(e_j)\left(U_j+W'_j\right)\subseteq  W_j+W'_j=W_j.
	\end{equation*} 

	Next, we prove part (ii.).  Let $1\leq j\leq r$. First,
	\begin{equation*}
		\sigma(0e_j)W=W.
	\end{equation*}
	Let $0\neq n\in \bbz$. By \cref{sigma_e_j_W} we have
	\begin{align}
	\sigma(n e_j)W&\subseteq\sigma((n-\mathrm{sgn}(n))e_j)\left(W+W_j\right)\label{sigm_e_j_ne_j1}\\
	&=\sigma((n-\mathrm{sgn}(n))e_j)W+W_j \label{sigm_e_j_ne_j2}\\
	&= \ldots=\sigma(e_j)\left(W+W_j\right)\subseteq\sigma(e_j)W+W_j,\label{sigm_e_j_ne_j3}
	\end{align}
	where \cref{sigm_e_j_ne_j2} is due to the $Z(R)$-invariant property of $W_j$, and \cref{sigm_e_j_ne_j3} is followed by repeating \cref{sigm_e_j_ne_j1,sigm_e_j_ne_j2} $n$ times.
	Thus, by taking $W'=\sum_{j=1}^r W_j$ the statement readily follows.
	
	In order to deduce part (iii.) we first note that $\sigma$ is smooth, so for all $v\in L$, the space $\mathrm{sp}\{\sigma(x)v|\ x\in K\}$ is of finite dimension. Hence, for a finite dimensional space $W\subseteq L$, the space 
	\begin{equation*}
	\sigma(K)W=\mathrm{sp}\{\sigma(x)w|\ x\in K, w\in W\}
	\end{equation*}
	is also of finite dimension. It follows that $\sigma(K)\bbc v$ is finite dimensional. Therefore, by part 2 there exists $W'\subseteq L$ a finite dimensional $Z(R)$-invariant space such that
	\begin{equation*}
		\mathrm{sp}_{\bbc}\{\sigma(z)v|\ z\in Z(R)\}= \sigma(\bbz^r)\left(\sigma(K)\bbc v\right)\subseteq \sigma(K)\bbc v+W'.
	\end{equation*}
\end{proof}

We are now ready to prove \Cref{prop:The-space-spanned}.
\begin{proof}[Proof of \Cref{prop:The-space-spanned}]
The Jacquet module is a smooth $G$-module of finite length \cite[Theorems 3.3.1 and 6.3.10]{casselman1993introduction}. 
Applying part (iii.) of \Cref{seclem:cent_gen_fin}  for $R=M_\Theta$ (with $Z(R)=T_\Theta$), $L=V_{N_\Theta}$, $\sigma=\pi_{N_\Theta}$, and $v=u_\Theta\in V_{N_\Theta}$, gives that $\left\{ \pi_{N_\Theta}\left(b\right)u_\Theta|\ b\in T_{\Theta}\right\} $ is of finite dimension. Let $\left\{  \pi_{N_\Theta}\left(b_{1}\right)u_\Theta,\ldots, \pi_{N_\Theta}\left(b_{\ell}\right)u_\Theta\right\}$
be its basis. Let $b\in T_\Theta$ and $m\in M_\Theta$. Then,
\begin{equation*}
\pi_{N_\Theta}\left(b\right)u_\Theta=\sum_{i=1}^{\ell}c_{i}(b)\pi_{N_\Theta}\left(b_i\right)u_\Theta.
\end{equation*}
Therefore,
\begin{equation*}
\left\langle \pi_{N_\Theta}\left(mb\right)u_\Theta,\tilde{u}_\Theta\right\rangle _{N_\Theta}=\sum_{i=1}^{\ell}c_{i}\left(b\right)\left\langle \pi_{N_\Theta}\left(m b_i\right)u_\Theta,\tilde{u}_\Theta\right\rangle _{N_\Theta}.
\end{equation*}
\end{proof}

\section{General linear groups}\label{gln_case_sec}

Let $n$ be a positive integer.
We provide a proof of \Cref{torus_decomp_lem} for the $\mathrm{GL}_n$ case. This will allow us to express  $b\in T_{\Theta}^{-}\left(\varepsilon\right)$, corresponding to $a\in T_\phi^-$ in the decomposition \eqref{cone_decomp}, in terms of the coordinates of $a$.
\begin{proof}[Proof of \Cref{torus_decomp_lem} (for $\mathrm{GL}_n$)]
	Let $0<\varepsilon<1$, $\Theta\subseteq \Delta$, and $a\in {_{\Theta}T_{\phi}^{-}}\left(\varepsilon\right)$. There exists a positive integer $\ell$ such that  $T_\phi^-\cong \left(k^\times\right)^\ell$. Under this isomorphism we identify $a$ with $(a_1,\ldots,a_\ell)$. 
	We denote $\Delta=\{\alpha_1,\ldots,\alpha_{\ell-1}\}$, where for all $1\leq i \leq \ell-1$, $\alpha_i(a)=\frac{a_i}{a_{i+1}}$. Denote $I_\Theta=\{i_1,\ldots,i_m\}\subseteq \{1,\ldots,\ell-1\}$, such that $\Theta=\{\alpha_{i_1},\ldots,\alpha_{i_m}\}$. We write $\{1,\ldots,\ell-1\}\backslash I_\Theta =\{i_{m+1},\ldots,i_{\ell-1}\}$. Thus, 
	\begin{equation}\label{roots_valuation_gln}
	\begin{cases}
		\varepsilon<|\frac{a_{i_j}}{a_{i_j+1}}|\leq 1,&\forall 1\leq j\leq m,\\
		|\frac{a_{i_j}}{a_{i_j+1}}|\leq \varepsilon,&\forall m+1\leq j \leq \ell-1.
	\end{cases}
	\end{equation}
	We take $a'$ identified with $(a'_1,\ldots,a'_\ell)$ such that for all $1\leq i\leq \ell$, the following change of variables is satisfied $a_i=\prod_{j=i}^{\ell}a'_j$.
	We have $\alpha_i(a)=a'_i$. Therefore, \eqref{roots_valuation_gln} gives
	\begin{equation}\label{roots_valuation_gln_var_change}
	\begin{cases}
		\varepsilon<|a'_{i_j}|\leq 1,&\forall 1\leq j\leq m,\\
		|a'_{i_j}|\leq \varepsilon,&\forall m+1\leq j \leq \ell-1.
	\end{cases}
	\end{equation}

	We take $b$ identified with $(b_1,\ldots,b_\ell)$ as follows. For $1\leq i\leq \ell$, 
	\begin{equation*}
		b_i= a'_\ell\prod_{\substack{i_j\geq i\\m < j < \ell}} a'_{i_j}.
	\end{equation*}
	Note that $b_\ell=a'_\ell=a_\ell$.
	Let $1\leq j'\leq m$. Then,
	\begin{equation*}
		\alpha_{i_{j'}}(b)=\frac{b_{i_{j'}}}{b_{i_{j'}+1}}=\frac{a'_\ell\prod\limits_{\substack{i_j\geq i_{j'}\\m < j < \ell}} a'_{i_j}}{a'_\ell\prod\limits_{\substack{i_j\geq i_{j'}+1\\m < j < \ell}} a'_{i_j}}=1,
	\end{equation*}
	i.e. $\alpha(b)=1$ for all $\alpha\in \Theta$.
	 Let $m < j' <\ell$. Then,
	\begin{equation*}
	\alpha_{i_{j'}}(b)=\frac{b_{i_{j'}}}{b_{i_{j'}+1}}=\frac{a'_\ell\prod\limits_{\substack{i_j\geq i_{j'}\\m < j < \ell}} a'_{i_j}}{a'_\ell\prod\limits_{\substack{i_j\geq i_{j'}+1\\m < j < \ell}} a'_{i_j}}=a'_{i_{j'}},
	\end{equation*}
	i.e. $|\alpha(b)|\leq \varepsilon$ for all $\alpha\in\Delta \backslash \Theta$.
	Thus, $b\in T_{\Theta}^{-}\left(\varepsilon\right)$.
	
	We take $c=b^{-1} a$, so $c_\ell=1$ and for all $1\leq i \leq \ell-1$,
	\begin{equation*}
		c_i=\frac{\prod\limits_{i\leq j \leq \ell}a'_j}{a'_\ell\prod\limits_{\substack{i_j\geq i\\m < j < \ell}} a'_{i_j}}=\prod_{\substack{i\leq i_j< \ell \\1 \leq j \leq m}} a'_{i_j}.
	\end{equation*} 
		
	By \Cref{roots_valuation_gln_var_change}, for all $1\leq j\leq m$ we have $\varepsilon<|a'_{i_j}|\leq 1$. Thus,  $c_i$ is bounded for all $1\leq i \leq \ell-1$ as a product of bounded elements.
	For all $1\leq i< \ell$, we write $c_i=u_i\pomega^{r_i}$ where $u_i\in \ocal^\times$, and $r_i$ is an integer in a bounded set. Next, we split $c=c_\pomega c_u$, such that $c_u$ is identified with $(u_1,\ldots,u_{\ell-1},1)$ and $c_\pomega$ is identified with $\left(\pomega^{r_1},\ldots,\pomega^{r_{\ell-1}},1\right)$.
	
	The stabilizer of $v$ is a congruence subgroup of $K=\mathrm{GL}_n(\ocal)$, the maximal compact subgroup of $G$. Hence, $K_v\leq K\cap T_1\cong \left(\ocal^\times\right)^n$, and as an open subgroup, it has a finite index in this compact group. i.e., there exists $d$ and $\{x_i\}_{i=1}^{d}$ such that
	\begin{equation*}
		K\cap T_1=\bigcup_{i=1}^{d} x_i K_v.
	\end{equation*}
	Now $c_u \in K\cap T_1$. Therefore, 
	\begin{equation*}
	c=c_\pomega c_u \in \bigcup_{i=1}^d c_\pomega x_i K_v \subseteq \bigcup_{\gamma\in \Gamma_\Theta^\varepsilon} \gamma K_v,
	\end{equation*}
	where $\Gamma_\Theta^\varepsilon$ is a finite set consisting of all the (finitely many) possibilities for $c_\pomega$ multiplied by all the finitely many representatives $x_i$.
	Hence,
	\begin{equation*}
	a=bc\in  T_{\Theta}^{-}\left(\varepsilon\right)\bigcup_{\gamma\in \Gamma_\Theta^\varepsilon} \gamma K_v.
	\end{equation*}	
\end{proof}

From this proof we conclude that for a given $a\in  {_{\Theta}T_{\phi}^{-}}\left(\varepsilon\right)$ identified with $(a_1,\ldots,a_\ell)$, $\Theta=\{\alpha_{i_1},\ldots,\alpha_{i_m}\}$, and $\Delta\backslash \Theta= \{\alpha_{i_{m+1}},\ldots\alpha_{i_{\ell-1}}\}$, the corresponding $b\in  T_{\Theta}^{-}\left(\varepsilon\right)$ from \Cref{torus_decomp_lem} is a block diagonal matrix with scalars blocks where the different scalars are
\begin{equation*}
	\left\{b_{i_{j'}}\right\}_{m<{j'}<\ell}\cup\{a'_\ell\}=\left\{a'_\ell\prod\limits_{\substack{i_j\geq i_{j'}\\m < j < \ell}} a'_{i_j}\right\}_{m<{j'}<\ell}\bigcup\{a'_\ell\},
\end{equation*}
where for all $1\leq i \leq \ell$, $a_i=\prod_{j=i}^{\ell}a'_j$, or equivalently $a'_\ell=a_\ell$ and for all $1\leq i \leq \ell-1$, $a'_i=a_{i+1}^{-1}a_i$.
Therefore, the different scalars in $b$ are
\begin{equation*}
\left\{a_\ell\prod\limits_{\substack{i_j\geq i_{j'}\\m < j < \ell}} a_{i_j+1}^{-1}a_{i_j}\right\}_{m<{j'}<\ell}\bigcup\{a_\ell\}.
\end{equation*}

By the multiplicity of characters and absolute values, and by the additivity of logarithm, we can write \eqref{asym_exp_p_adic} explicitly.
\begin{cor}
	Let $G=\mathrm{GL}_n$. In the setting of \Cref{asym_exp_for_all_prop}, and by writing $\Delta\backslash\Theta=\{\alpha_{i_{r+1}},\ldots,\alpha_{i_{n-1}}\}$, formula \eqref{asym_exp_p_adic} takes the following form.
	\begin{equation}
		\left\langle \pi(a)v,\tilde{v}\right\rangle = \sum_{\underline{p'},\underline{p},\underline{\chi}}\alpha_{\underline{p'},\underline{p},\underline{\chi}}\chi_{n-r}(a_n)\left|a_n\right|^{p'_{n-r}}\log_q^{p_{n-r}}\left|a_n\right|\prod_{j=r+1}^{n-1}\chi_{j-r}(a_{i_j})\left|a_{i_j}\right|^{p'_{j-r}}\log_q^{p_{j-r}}\left|a_{i_j}\right|.
	\end{equation}
\end{cor}

\section*{Acknowledgment}
I would like to express my appreciation to David Soudry and Elad Zelingher for useful discussions and their valuable comments on earlier versions of this manuscript.
\bibliographystyle{alpha}
\bibliography{references}

\Addresses
\end{document}